\documentclass[12pt]{amsart}

\usepackage[foot]{amsaddr}
\author{Colin Crowley$^1$} 
\email{cwcrowley@wisc.edu} 
\address{$^1$UW-Madison Department of Mathematics, Van Vleck Hall, 480 Lincoln Drive, Madison, WI 53706} 
\author{Noah  Giansiracusa$^2$} 
\email{ngiansi1@swarthmore.edu}
\address{$^2$Department of Mathematics and Statistics, Swarthmore College, 500 College Ave, Swarthmore, PA 19081}
\author{Joshua Mundinger$^3$} 
\email{mundinger@uchicago.edu}
\address{$^3$Department of Mathematics, University of Chicago, 5734 S. University Avenue, Room 108, Chicago, IL 60637}

\usepackage{etex}
\usepackage{amsmath, amsthm, bm, graphicx, float, amssymb,verbatim,color,xypic,mathrsfs,enumerate}
\usepackage[all]{xy}
\usepackage[margin=1.0in]{geometry}
\usepackage{mathptmx}
\ifx\JPicScale\undefined\def\JPicScale{1.0}\fi
\unitlength \JPicScale mm
{\theoremstyle{plain}
   
   \newtheorem{theorem}{Theorem}[section]
   \newtheorem{proposition}[theorem]{Proposition}     
   \newtheorem{lemma}[theorem]{Lemma}
   
   \newtheorem{corollary}[theorem]{Corollary}

}
{\theoremstyle{definition}
   \newtheorem*{ack}{Acknowledgements}
   
   \newtheorem{example}[theorem]{Example}
   \newtheorem{definition}[theorem]{Definition}
   \newtheorem{remark}[theorem]{Remark}
   \newtheorem{conv}[theorem]{Convention}
   
}

\newcommand{\B}{\mathbb{B}}

\newcommand{\R}{\mathbb{R}}

\newcommand{\T}{\mathbb{T}}
\newcommand{\supp}{\mathrm{supp}}

\newcommand{\Sym}{\mathrm{Sym}}

\newcommand{\GL}{\operatorname{GL}}

\newcommand{\Gr}{\operatorname{Gr}}

\newcommand{\Hom}{\operatorname{Hom}}
\newcommand{\End}{\operatorname{End}}

\newcommand{\im}{\operatorname{Im}}
\newcommand{\id}{\operatorname{id}}
\newcommand{\cl}{\operatorname{cl}}

\newcommand{\surj}{\twoheadrightarrow}

\newcommand{\tropker}{\mathrm{tropker}}
\newcommand{\stsum}{+^{st}}
\newcommand{\dual}{\vee}
\newcommand{\ddual}{{\dual\dual}}
\DeclareMathOperator{\rk}{rk}
\newcommand{\bend}[1]{\mathscr{B}\left(#1\right)}
\newcommand{\Bdual}[1]{(\B^{#1})^\dual}

\numberwithin{theorem}{section}

\newcommand{\Ifandonlyif}[2]{$\text{(\ref{#1})}\iff\text{(\ref{#2})}$}

\begin{document}
\title{A module-theoretic approach to matroids}

\begin{abstract}
Speyer recognized that matroids encode the same data as a special class of tropical linear spaces and Shaw interpreted tropically certain basic matroid constructions; additionally, Frenk developed the perspective of tropical linear spaces as modules over an idempotent semifield.  All together, this provides bridges between the combinatorics of matroids, the algebra of idempotent modules, and the geometry of tropical linear spaces.  The goal of this paper is to strengthen and expand these bridges by systematically developing the idempotent module theory of matroids.  Applications include a geometric interpretation of strong matroid maps and the factorization theorem; a generalized notion of strong matroid maps, via an embedding of the category of matroids into a category of module homomorphisms; a monotonicity property for the stable sum and stable intersection of tropical linear spaces; a novel perspective of fundamental transversal matroids; and a tropical analogue of reduced row echelon form.
\end{abstract}

\maketitle

\section{Introduction}

\subsection{Overview}
Matroids, first introduced by Whitney in 1935, provide a combinatorial abstraction of linear dependence.  They have a remarkable tendency to appear in diverse settings across mathematics, pure and applied \cite{welsh76,lawler76,white86,recski89,white92,BGW03,murota10,oxley11,pitsoulis14}.  The first main bridge to algebraic geometry came in the 80s with the advent of matroid polytopes, linking matroids to toric varieties and the Grassmannian \cite{GGMS87,GS87,kapranov-chow}.  A more recent connection to algebraic geometry stems from the work of Speyer and Sturmfels showing that valuated matroids, a generalization of matroids in which each basis carries a real number, are equivalent to tropical linear spaces \cite{speyer08,speyer-sturmfels,maclagan-sturmfels,fink-chow}.  In his PhD thesis, Frenk carefully studied the module theory underlying this latter connection \cite{frenk13}, and this was used in \cite{giansiracusa17} to introduce an idempotent exterior algebra reinforcing Speyer's philosophy that valuated matroids are tropical Pl\"ucker vectors. 

 In this paper, we step back from valuated matroids to ordinary matroids so that more of the extensive literature on matroids can be brought into the realm of idempotent modules and tropical linear spaces.  In doing so, we clarify the relation between matroids and tropical linear algebra, gain insight into matroids by exploiting this module-theoretic perspective, and gain insight into tropical linear spaces by drawing more heavily from matroid theory.

\subsection{Summary of results}
Let $\B = \{0,1\}$ denote the two-element idempotent semifield; that is, $\B$ satisfies all the axioms of a field except for the existence of additive inverses, and the ``characteristic one'' condition $1+1=1$ holds.  Associated to a matroid $M$ on the ground set $E$ is a submodule $L_M \subseteq \B^E$ spanned by the indicator vectors of the cocircuits, and a quotient module $(\B^E)^\vee \twoheadrightarrow Q_M$ introduced in \cite{giansiracusa17} satisfying $Q_M^\vee = L_M$, where the superscript $\vee$ denotes the dual module obtained by applying the functor $\Hom(-,\B)$.  The module $L_M$ (resp. $Q_M$) doesn't uniquely determine the matroid $M$, but with the extra data of its embedding in $\B^E$ (resp. quotient presentation) it does.  If $\T = \R\cup\{-\infty\}$ denotes the tropical semifield, then $L_M \otimes_\B \T \subseteq \T^n$ is the tropical linear space associated to $M$, so by slight abuse of terminology we shall refer to $L_M$ itself as a tropical linear space (and we call tropical linear spaces of this form \emph{constant-coefficient} tropical linear spaces).  

\begin{theorem}\label{thm-intro:dual}
The quotient module $Q_M$ is isomorphic to the lattice of flats of $M$, and it is the module-theoretic dual of the corresponding tropical linear space: $Q_M = L_M^\vee$.
\end{theorem}

Most aspects of this result already appear in the literature (see \cite[Proposition 4.4 and Remark 4.5]{maclagan-rincon}, \cite[Theorem 1]{ardila-klivans}, and \cite{frenk13}) but we provide a streamlined proof firmly rooted in matroid theory.   We also give a direct matroid-theoretic proof of the following, which Frenk established more generally for valuated matroids \cite[Lemma 4.1.11]{frenk13}:

\begin{theorem}\label{thm-intro:minors}
For a subset $T\subseteq E$, the tropical linear space of the matroid contraction $M/T$ is \[L_{M/T} = L_M\cap \B^{E-T},\] where $\B^{E-T} \subseteq \B^E$ is a coordinate subspace, and the tropical linear space of the deletion $M\backslash T$ is \[L_{M\backslash T} = \pi_{E-T}(L_M),\] where $\pi_{E-T} : \B^E \rightarrow \B^{E-T}$ is coordinate linear projection.
\end{theorem}

Since minors are iterations of the contraction and deletion operations, this provides a geometric interpretation of the tropical linear space associated to any minor of a matroid (see also \cite[Proposition 2.22]{shaw13}).  By Theorem \ref{thm-intro:dual}, a similar statement holds for the quotient modules $Q_M$.

Our module-theoretic perspective leads to a new characterization of strong matroid maps in terms of tropical linear spaces (stated here in a slightly weakened version for readability---in \S\ref{subsec:strong} we allow for a distinguished loop):

\begin{theorem}\label{thm-intro:strong}
Let $M$ and $N$ be matroids on ground sets $E$ and $F$, respectively, let $f: E \rightarrow F$ be any function, and define $f_* : (\B^E)^\vee \rightarrow (\B^F)^\vee$ by $f_*(x_i) = x_{f(i)}$, where $\{x_i\}$ is the dual basis to the standard basis.  Then $f$ is a strong matroid map $M \rightarrow N$ if and only if $f_*^\vee(L_N) \subseteq L_M$.
\end{theorem}

In the special case where $E=F$ and $f$ is the identity map, a strong matroid map is the same as a matroid quotient, and in this case the preceding result is due to Shaw \cite[Lemma 2.21]{shaw13}.  The factorization theorem for strong maps \cite[Theorem 8.2.7]{white86}, which states that every strong matroid map factors as an embedding followed by a contraction, 
in this case translates---via Theorems \ref{thm-intro:minors} and \ref{thm-intro:strong}---to the following property of constant-coefficient tropical linear spaces:
\begin{theorem}\label{thm-intro:existsL}
Let $L$ and $L'$ be tropical linear spaces in $\B^E$.  Then $L \subseteq L'$ if and only if there exists a tropical linear space $L''\subseteq \B^{E\sqcup F}$ with $|F|=\dim(L')-\dim(L)$ such that $L= L'' \cap \B^E$ and $L' = \pi_{E}(L'')$.
\end{theorem}
This is a direct analogue of a property of linear spaces over a field.  Moreover, when $L\subseteq L'$, the tropical linear space $L''$ is a tropical modification of $L'$ along $L$ (cf. \cite[\S2.4]{shaw13}), so in essence the factorization theorem for strong maps with $\id : E \rightarrow E$ is the existence of arbitrary-codimension tropical modifications in the case of constant-coefficient degree one tropical varieties.

It has been recognized that there are problems with the category of matroids equipped with strong maps---primarily, it lacks important limits and not enough matroidal constructions are functorial \cite{alhawary01,heunen17}.  Module theory provides a potential partial remedy:
\begin{theorem}
The category of matroids with strong maps embeds faithfully as a (non-full) subcategory of the category of quotients $(\B^E)^\vee \twoheadrightarrow Q_M$ with morphisms given by commutative squares:
\[\xymatrix{(\B^E)^\vee \ar@{->>}[d] \ar[r] & (\B^F)^\vee \ar@{->>}[d]\\ Q_M \ar[r] & Q_N.}\]  It also embeds contravariantly in the category of embedded tropical linear spaces $L_M \subseteq \B^E$ with morphisms given by $\B$-module maps $\B^E \rightarrow \B^F$ that restrict to $\B$-module maps $L_M \rightarrow L_N$.
\end{theorem}
These categories expand the class of strong maps by allowing multi-valued functions on the ground sets, since basis vectors can be sent to sums of basis vectors.  While we have not yet explored these categories in depth, it is a reasonable hope that multi-valued strong maps have nice matroidal properties and that some important categorical limits can be added by taking the limits after these embeddings---but this is a topic for future work.  For now, we include a simple example of the compatibility between matroid theory and module theory exhibited by these categories:
\begin{theorem}
The functor $M \rightsquigarrow \left((\B^E)^\vee \twoheadrightarrow Q_M\right)$ sends direct sum decompositions (in the matroidal sense) to direct sum decompositions (in the categorical sense), and every categorical direct sum arises in this way.  In particular, $M$ is connected if and only if $(\B^E)^\vee \twoheadrightarrow Q_M$ is indecomposable.
\end{theorem}

If $M$ and $M'$ are matroids of ranks $d$ and $d'$, respectively, on the ground set $E$, and if the matroid union $M\vee M'$ has rank $d+d'$, then the tropical linear space $L_{M\vee M'} \subseteq \B^E$ is the \emph{stable sum} of $L_M$ and $L_{M'}$, an operation dual to the stable intersection \cite{fink15}.  We give a geometric proof, based on Theorems \ref{thm-intro:minors} and \ref{thm-intro:existsL}, of the following basic monotonicity property which is a direct translation of a known property of matroid quotients:

\begin{theorem}
If $L_{M_1} \subseteq L_{M_2} \subseteq \B^E$ are tropical linear spaces and $M$ is a matroid on $E$ such that $\rk(M\vee M_2) = \rk(M)+\rk(M_2)$, then $L_{M\vee M_1} \subseteq L_{M \vee M_2}$.
\end{theorem}

By matroidal/tropical duality, an analogous statement holds for the stable intersection.  For valuated matroids supported on the uniform matroid, this monotonicity property is an immediate consequence of Speyer's analytic interpretation of the stable intersection as a limit of perturbations \cite{speyer08}; Speyer's framework can be extended to the generality of our result, but our module-theoretic argument nonetheless provides a new perspective here.

Transversal matroids are those that factor into a union of rank-one matroids.  Geometrically, this means their tropical linear spaces are stable sums of lines \cite{fink15}.  Factoring a rank $d$ transversal matroid $M$ on $E$ into a union of rank-one matroids is equivalent to writing a $d\times |E|$ matrix over $\B$ such that the subsets of columns yielding nonzero maximal minors are precisely the bases of $M$.  Let us call such a matrix a \emph{$\B$-presentation} of $M$.  Combinatorial results in \cite{bondy-welsh,bondy72} translate immediately into the following geometric statement:

\begin{theorem}
Each nonempty fiber of the constant-coefficient tropical Stiefel map \cite{fink15} \[\B^{d\times |E|} \dashrightarrow \Gr^{trop}(d,|E|)\]  sending a $\B$-matrix to its tropical Pl\"ucker vector has a unique maximal member and in general multiple distinct minimal members.
\end{theorem}

The maximal elements of these fibers can, in a sense, be viewed as a tropical analogue of reduced row echelon form, and the process of transforming an arbitrary $\B$-presentation into the maximal one a tropical analogue of Gauss-Jordan elimination---although reduced echelon matrices tropicalize to minimal, rather than maximal, elements of these fibers, so perhaps ``augmented'' form is a better term than ``reduced'' form (and one is reminded here of the moduli space of curves  where the maximal and minimal strata are reversed when passing from classical to tropical moduli \cite{ACP}).

A special, and important \cite{bixby77,recski-iri,bonin10}, class of transversal matroids are the fundamental transversal matroids.  Translating a known characterization of these to our module-theoretic setting reads as follows:

\begin{theorem}
A rank $d$ transversal matroid on $E$ is fundamental transversal if and only if it admits a $\B$-presentation by a matrix inducing a surjective linear map $\B^{|E|} \twoheadrightarrow \B^d$.
\end{theorem}

Here is one perspective of this observation.  Over a field $k$, every rank $d$ linear subspace of $k^n$ corresponds to the maximal minors of a surjective (i.e., full rank) $d\times n$ matrix; over $\B$, the rank $d$ tropical linear spaces in $\B^n$ which correspond to the maximal minors of a $d\times n$ matrix are the transversal matroids \cite{fink15,giansiracusa17} and the ones corresponding to a surjective matrix are the fundamental transversal matroids.

\ack{We thank Alex Fink, Jeff Giansiracusa, Andrew Macpherson, Felipe Rinc\'on, and Bernd Sturmfels for helpful conversations.  The second author was supported in part by NSA Young Investigator Grant H98230-16-1-0015 and NSF DMS-1802263.}

\section{Preliminaries}

In this section we recall some background material that will be relied upon throughout the paper.

\subsection{Matroids}

What follows is basic terminology and elementary facts about matroids that can be found in any of the standard texts on matroids, e.g. \cite{oxley11}; all concepts and results from matroid theory beyond this will be discussed as needed throughout the paper.

A matroid can be defined in many distinct, ``cryptomorphic'' ways.  Most relevant for us is the definition in terms of bases: a \emph{matroid} on the \emph{ground set} $E = \{1,2,\ldots,n\}$ is a collection of subsets $\mathcal{B} \subseteq 2^E$, called the \emph{bases}, such that (B1) $\mathcal{B} \ne \emptyset$ and (B2) if $B_1, B_2 \in \mathcal{B}$ and $i\in B_2 - B_1$ then there exists $j\in B_1 - B_2$ such that $(B_2-i)\cup j \in \mathcal{B}$.  Axiom (B2) is called the \emph{basis exchange axiom}; it can be replaced by the following \emph{strong basis exchange axiom} without changing the class of matroids  thus defined: (B2$'$) if $B_1, B_2 \in \mathcal{B}$ and $i\in B_2 - B_1$ then there exists $j\in B_1 - B_2$ such that $(B_2-i)\cup j \in \mathcal{B}$ and $(B_1-j)\cup i \in \mathcal{B}$.  All bases of a matroid have the same cardinality, and this is the \emph{rank} of the matroid.  

The subsets of bases are called \emph{independent sets}, subsets that are not independent are \emph{dependent}, the minimal dependent sets are the \emph{circuits}, and unions of circuits are \emph{cycles}.  If $B$ is a basis and $e\in E-B$, then $B\cup e$ contains a unique circuit, denoted $C(e,B)$ and called the \emph{fundamental circuit of $e$ with respect to $B$}; we have $e\in C(e,B)$, and every circuit of $M$ is a fundamental circuit for some $B$ and $e$ \cite[Corollary 1.2.6]{oxley11}.  The subsets of $E$ given by the complements of the bases of a rank $d$ matroid $M$ form the collection of bases of the \emph{dual matroid} $M^*$, which has rank $n-d$.  The prefix ``co'' in matroid theory refers to the dual matroid, so for instance the \emph{cocircuits} of $M$ are the circuits of $M^*$.  

If $T\subseteq E$ is any subset of the ground set and $M$ is a matroid on $E$, then the \emph{restriction of $M$ to $T$}, denoted $M|T$, is the matroid whose independent sets are the independent sets of $M$ which are contained in $T$.  The \emph{rank} of $T$ in $M$ is by definition the rank of the matroid $M|T$.  The \emph{closure} of $T \subseteq E$ is the set $\cl(T) := \{i\in E : \rk(T\cup i) = \rk(T)\}$, where $\rk : 2^E \rightarrow \mathbb{N}$ denotes the rank function.  A \emph{flat} is a subset that equals its own closure, i.e., it is maximal for its rank.  On multiple instances we use the elementary observation that complements of flats are precisely the unions of cocircuits (see., e.g., the proof of \cite[Proposition 5.3.19]{maclagan-sturmfels}).  The set of flats of $M$, partially ordered by inclusion, forms a lattice denoted $\mathcal{L}(M)$ and called the \emph{lattice of flats}.

A 1-element circuit is a \emph{loop} and a 2-element circuit is a \emph{parallel edge}.  A matroid is \emph{simple} if it has no loops or parallel edges, and every matroid admits a unique \emph{simplification}.  The lattice of flats of a matroid determines the simplification of the matroid, and vice versa.

\begin{conv}
Throughout the paper $M$ is a rank $d$ matroid on the ground set $E = \{1,2,\ldots,n\}$.
\end{conv}

\subsection{Idempotent semirings and modules}

A \emph{semiring} is a set with two binary operations that satisfy all the axioms of a ring except for the existence of additive inverses \cite{golan99}.   A \emph{homomorphism} of semirings is a set map compatible with the two operations, and a \emph{module} $N$ over a semiring $S$ is an abelian monoid $N$ together with a semiring homomorphism $S \rightarrow \End(N)$.  Homomorphisms of $S$-modules, isomorphisms, and $S$-submodules are defined in the obvious way.  We denote the \emph{dual} of an $S$-module $N$ by $N^\vee = \Hom(N,S)$.  

Quotients of semirings and their modules are a bit different than in the setting of rings, since the lack of additive inverses means one cannot identify a pair of elements, say $f\sim g$, by identifying their difference with zero, $f-g\sim 0$. A \emph{congruence} on an $S$-module $N$ is an equivalence relation that is also a submodule of $N\times N$.  This means we are specifying which pairs of elements of $N$ to identify in a quotient, and doing so in a way that ensures the $S$-module structure descends to the quotient.  Given any subset $A\subseteq N\times N$ there is a unique smallest congruence containing $A$, and this is the congruence \emph{generated} by $A$.  A \emph{semiring congruence} on $S$ is an equivalence relation on $S\times S$ that is also a subsemiring, and an $S$-algebra congruence is defined similarly.  The reader may consult, e.g., \cite[\S2.4]{GG1} for more on these topics.  

A semiring $S$ is \emph{idempotent} if $1+1=1$ (or equivalently, $s+s=s$ for all $s\in S$).  There is a unique 2-element idempotent semiring, the \emph{Boolean semiring} $\B = \{0,1\}$.  Explicitly, in $\B$ we have \[0+0=0,~0+1=1,~1+1=1,\]
\[0\cdot 0 = 0,~0\cdot 1 = 0,~1\cdot 1 = 1.\]
The Boolean semiring $\B$ is the initial object in the category of idempotent semirings, so every idempotent semiring is canonically a $\B$-algebra.  The following observation will be important when we study the lattice of flats of a matroid from an algebraic perspective:
\begin{remark}
Every finitely generated $\B$-module is a finite lattice, and conversely every finite lattice is a finitely generated $\B$-module \cite[Proposition 3.1.1 and Corollary 3.1.3]{frenk13}.  A homomorphism of $\B$-modules is an isomorphism if and only if it is a lattice isomorphism.
\end{remark}
Another important idempotent semiring for us is the \emph{tropical numbers} $\T = \R\cup\{-\infty\}$ equipped with addition given by the maximum and multiplication given by the ordinary addition.  To view the Booleans as a subsemiring in $\T$ we must write $\B=\{-\infty,0\}$; to avoid the potential confusion resulting from this we prefer to view $\T$ as an abstract idempotent semiring and write its neutral elements as $0$ and $1$.  

For a finite set $E = \{1,2,\ldots,n\}$ we denote by $\B^E$ the free module of rank $n$ with standard basis $\{e_i\}_{i=1}^n$.  (In fact, up to permutation this basis is unique \cite[Proposition 2.2.2]{giansiracusa17}.)  We denote the dual basis for $(\B^E)^\vee$ by $\{x_i\}_{i=1}^n$.  

\begin{definition}
	The \textit{bend relations} of a linear form $f = \sum_i f_i x_i \in \Bdual{E}$ are the relations
	\begin{equation*}
		\left\{ f \sim \sum_{i \neq j} f_i x_i : j \in E\right\}.
	\end{equation*}
	We write $\bend{f} \subseteq \Bdual{E}\times \Bdual{E}$ for the congruence generated by the bend relations of $f$.
\end{definition}

Bend relations were originally introduced in \cite{GG1} to define a scheme structure on tropicalization.  In the simplest case,  for a tropical polynomial $f\in \T[x_1,\ldots,x_n]$, the $\T$-points of the quotient by (the $\T$-algebra congruence generated by) the bend relations of $f$ form the tropical hypersurface $V^{trop}(f) \subseteq \T^n$, and this quotient itself is the coordinate algebra of this tropical hypersurface.  In the present paper we shall only need the linear variant of this construction formalized in \cite{giansiracusa17} and stated below.  First, more notation: for a collection of linear forms $A \subseteq (\B^E)^\dual$ we write $\bend{A}$ for the congruence generated by the bend relations of all $f \in A$; if $\varphi: \B^E \to N$ is a $\B$-module homomorphism, then we have the dual transformation $\varphi^\dual: N^\dual \to \left(\B^E\right)^\vee$ and let $\bend{\varphi} := \bend{\im\varphi^\vee}$.
\begin{definition}
	The \textit{tropical kernel} of a $\B$-module homomorphism $\varphi: \B^E \to N$ is 
	\[\tropker(\varphi) := \left\{\sum_i c_i e_i \in \B^E : \varphi\left( \sum_i c_i e_i \right) = \varphi\left(\sum_{i \neq j} c_i e_i \right), \forall j \in E\right\} = \left(\Bdual{E} / \bend{\varphi}\right)^\vee \subseteq \B^E.\]  
\end{definition}

For a linear form $f : \B^E \rightarrow \B$, the $\T$-submodule $\tropker(f)\otimes_\B \T \subseteq \T^E$ is the tropical hyperplane defined by the piecewise linear function $f\otimes_\B\T : \T^E \rightarrow \T$.  For this reason, we shall call $\tropker(f)$ the \emph{tropical hyperplane} associated to $f\in (\B^E)^\vee$.  More generally, for $\varphi : \B^E \rightarrow N$ the $\T$-submodule $\tropker(\varphi)\otimes_\B\T \subseteq \T^E$ is an intersection of tropical hyperplanes, but it has been known since the early days of tropical geometry that only certain such intersections deserve to be called tropical linear spaces, as we next discuss.

\subsection{Tropical linear spaces}

By definition, a rank $d$ \emph{tropical Pl\"ucker vector} is a nonzero vector \[v=\sum_{I\in\binom{E}{d}}v_I e_I \in \T^{\binom{E}{d}},\] for $E = \{1,2,\ldots,n\}$, satisfying the \emph{tropical Pl\"ucker relations}, which from the perspective of bend relations means that for each $A\in \binom{E}{d+1}$ and $B\in \binom{E}{d-1}$ we have
\[\sum_{i \in A-B}v_{A-i}v_{B\cup i} = \sum_{i \in (A-B)- j}v_{A-i}v_{B\cup i}\] for each $j\in A-B$.  A tropical Pl\"ucker vector determines a submodule \[L_v := \bigcap_{A\in \binom{E}{d+1}}\tropker\left(\sum_{i\in A}v_{A-i}x_i\right) \subseteq \T^E\] called a \emph{tropical linear space}, and this submodule determines the tropical Pl\"ucker vector up to an additive constant in $\R$ \cite{speyer08,maclagan-sturmfels}.  

While a matroidal perspective on tropical Pl\"ucker vectors and tropical linear spaces is well-known \cite{speyer08,maclagan-sturmfels,frenk13}, we will not recall it here and instead go through it systematically in the body of this paper, filling in various details and significantly expanding the connection as we go.  For those who have not yet encountered this material, the basic point is that a tropical Pl\"ucker vector is exactly a valuated matroid, and the linear forms in the intersection defining $L_v$ correspond to the valuated circuits of this valuated matroid.  These assertions will be clarified later.  

\begin{remark}
We shall not need the following facts, but they add a nice geometric context: tropical linear spaces are precisely the degree one tropical varieties in $\T^n$ \cite[Theorem 7.4]{fink-chow}, and they are also precisely the tropical subvarieties of $\T^n$ that are simultaneously $\T$-submodules \cite{hampe}.  
\end{remark}

The punchline is that linear tropical geometry is valuated matroid theory.  We focus in this paper primarily on \emph{constant-coefficient tropical linear spaces}, meaning submodules $L_v\subseteq \B^E$ where the tropical Pl\"ucker vector $v$ lies in $\B^{\binom{E}{d}} \subseteq \T^{\binom{E}{d}}$, since matroids have been studied much more extensively than valuated matroids.

\subsection{Idempotent exterior algebra}\label{sec:idextalg}

A useful framework for tropical Pl\"ucker vectors and tropical linear spaces, which we draw from in this paper, was introduced in \cite{giansiracusa17}; we quickly recall the basics here, focusing on the constant-coefficient case.  

The symmetric algebra of the free module $\B^E$ is defined as a quotient of the tensor algebra in the usual way.  The \emph{tropical Grassmann algebra}, or \emph{exterior algebra}, is the $\B$-algebra $\bigwedge \B^E$ defined as the quotient of $\Sym\,\B^E$ by the $\B$-algebra congruence generated by $e_i^2\sim 0$.  The $d$-th graded piece, denoted $\bigwedge^d\B^E$, is free of rank $\binom{|E|}{d}$ with basis $\{e_I\}_{I\in \binom{E}{d}}$, where $e_I := e_{i_1}\wedge \cdots \wedge e_{i_d}$ when $I=\{i_1,\ldots,i_d\}$.  We use this basis to view the tropical Pl\"ucker relations as defining the locus of tropical Pl\"ucker vectors inside $\bigwedge^d\B^E \cong \B^{\binom{E}{d}}$.

Given a rank $d$ matroid $M$ with bases $\mathcal{B} \subseteq \binom{E}{d}$, we define the \emph{basis indicator vector} as  
\[[M] := \sum_{B\in \mathcal{B}}e_B \in \bigwedge^d\B^E,\] and $[M]_I$ denotes the $I$-th component of this multivector, which is $1$ if $I$ is a basis and 0 otherwise.  (If a matroid $N$ has rank zero, then $[N]=1$ in $\bigwedge \B^E$.)  The tropical Pl\"ucker vectors are then precisely the multivectors $v \in \bigwedge^d\B^E$ of the form $v=[M]$ for a matroid $M$, since the tropical Pl\"ucker relations are equivalent to the strong basis exchange axiom (B2$'$):
given bases $B_1$ and $B_2$ and $i \in B_2 - B_1$, we have $v_{(B_1\cup i)-i}v_{(B_2 - i)\cup i} = 1$, so the tropical Pl\"ucker relation for $A := B_1 \cup i$ and $B := B_2 - i$ provides a $j \in B_1 - B_2$ such that $v_{(B_1\cup i)-j}v_{(B_2-i)\cup j} = 1$ (the converse argument is similar).  We usually denote the tropical linear space associated to $[M]$ by $L_M$ rather than $L_{[M]}$.

Wedge-multiplication yields a $\B$-module homomorphism \[(-\wedge [M]) : \B^E \rightarrow \bigwedge^{d+1}\B^E,\] and
for each $A\in\binom{E}{d+1}$ the $A$-th component of this map is the linear form $\sum_{i\in A}[M]_{A-i}x_i$.  Since the intersection of the tropical hyperplanes defined by these linear forms yields the tropical linear space $L_M$, we have
\[\tropker(-\wedge [M]) = L_M \subseteq \B^E.\]
In \cite{giansiracusa17} the emphasis is moved from submodules to quotient modules, essentially since exterior powers are right exact functors; accordingly, the following quotient module is introduced:
\[Q_M := \Bdual{E}/\bend{-\wedge [M]}.\] 
Since the components of the linear map $(-\wedge [M])$ are the linear forms $\sum_{i\in A}[M]_{A-i}x_i$, the module $Q_M$ can be described concretely as the quotient by the bend relations of these linear forms.  By the definition of the tropical kernel, we have 
\begin{equation}\label{eq:Qdual}
Q_M^\vee = \tropker(-\wedge [M]) = L_M.
\end{equation}
A main result of \cite{giansiracusa17} is that the quotient map $\Bdual{E} \twoheadrightarrow Q_M$ induces a  rank one quotient \[\bigwedge^d\Bdual{E} \twoheadrightarrow \bigwedge^d Q_M \cong \B\] whose module-theoretic dual is the basis indicator vector $[M] \in \bigwedge^d\B^E$.  In general, the exterior algebra of a quotient $\B^E \twoheadrightarrow Q$ is defined as \[\bigwedge Q := \bigwedge\B^E\otimes_{\Sym\,\B^E}\Sym\,Q,\] which means we kill the squares of the images of the basis vectors.  

\begin{lemma}[{\cite[Lemma 3.2.2]{giansiracusa17}}] \label{lemma: wedge-power-relations}
	If $\B^E \twoheadrightarrow Q$ is generated by the relations $\{a_i \sim b_i\}$, then the quotient $\bigwedge^d \B^E \surj \bigwedge^d Q$ is generated by the relations $a_i \wedge e_I \sim b_i \wedge e_I$ for all $I \in \binom{E}{d-1}$.
\end{lemma}

\section{The algebraic and matroidal structure of $L_M$ and $Q_M$}\label{sec:QL}

The main goal of this section is to prove Theorem \ref{thm-intro:dual}, that for any matroid $M$ the quotient module $Q_M$ recalled in \S\ref{sec:idextalg} is isomorphic to the lattice of flats $\mathcal{L}(M)$ and dual to the associated tropical linear space, namely $Q_M = L_M^\vee$ (the equality $L_M = Q_M^\vee$, noted in Equation \eqref{eq:Qdual}, was already observed in \cite[Remark 4.3.3]{giansiracusa17}).  Along the way, we include proofs of a few facts about these modules previously established in the literature; by focusing on the constant-coefficient setting we are able to produce proofs that more closely integrate the module theory with matroid theory.  

\subsection{Matroidal interpretations of the tropical modules}

Recall that the tropical linear space $L_M \subseteq \B^E$ associated to the matroid $M$ is the intersection of tropical hyperplanes (i.e., tropical kernels) defined by the linear forms $\sum_{i\in A}[M]_{A-i}x_i \in \Bdual{E}$ for all $A\in\binom{E}{d+1}$, and the quotient module $Q_M$ is the quotient of the dual space $\Bdual{E}$ by the bend relations of these linear forms.  These linear forms are well-known to be the valuated circuits of $M$ when $M$ is a valuated matroid; we include here a quick proof of this fact in the present constant-coefficient setting:

\begin{lemma} \label{lemma: circuits}
	If $A\in \binom{E}{d+1}$ and $f_A := \sum_{i \in A} [M]_{A-i} x_i$ is nonzero, then $f_A = \sum_{i\in C}x_i$ for some circuit $C$ of $M$, and conversely for any circuit $C$ of $M$ we have $\sum_{i\in C}x_i = f_A$ for some $A\in \binom{E}{d+1}$.
\end{lemma}
\begin{proof}
	Suppose $f_A \ne 0$ for some $A \in \binom{E}{d+1}$.  Then $A = B \cup e$ for some basis $B$ and $e\in A-B$, so there is a unique circuit in $A$, namely the fundamental circuit $C(e,B)$.  This means $A-i$ is a basis if and only if $i\in C(e,B)$, so $f_A = \sum_{i \in C(e,B)} x_i$.  Conversely, any circuit is a fundamental circuit, say $C=C(e,B)$; then $\sum_{i\in C}x_i = f_A$ for $A={B\cup e}$.
\end{proof}

In particular, this means the congruence defining $Q_M$ is generated by the relations \[\sum_{i \in C-j} x_i \sim \sum_{i \in C} x_i\] for each circuit $C$ and $j \in C$.  This enables us to understand equality in the lattice $Q_M$, and thereby identify it with the lattice of flats $\mathcal{L}(M)$, using the following formulation of the closure function:
\begin{lemma}[{\cite[Proposition 1.4.11(ii)]{oxley11}}]
	For any subset $S\subseteq E$, the closure $\cl(S)$ is the union of $S$ with all elements $e \in E$
	such that there exists a circuit $C$ with $e \in C \subseteq S \cup e$.
\end{lemma}
\begin{theorem} \label{theorem: Q-is-flats}
	In the quotient module $Q_M$ we have
	\begin{equation*}
		\sum_{i \in S} x_i = \sum_{i \in S'} x_i
	\end{equation*}
	if and only if $\cl(S) = \cl(S')$.
	Consequently, $\mathcal{L}(M) \cong Q_M$ via $F \mapsto \sum_{i \in F} x_i$.
\end{theorem}
\begin{proof}
	First, we show that the relation $\sum_{i \in S} x_i \sim \sum_{i \in S'} x_i$ if and only if $\cl(S) = \cl(S')$ is a congruence on $\Bdual{E}$.  Since it is clearly an equivalence relation, we just need to check that it is a $\B$-submodule of $\Bdual{E}\times\Bdual{E}$, which in combinatorial terms means the following: if $\cl(S) = \cl(S')$ and $\cl(T) = \cl(T')$, then $\cl(S\cup T) = \cl(S'\cup T')$.  It suffices to verify that \[\cl(S \cup T) = \cl(\cl(S) \cup \cl(T)).\]  We have $S\subseteq \cl(S)$ and $T\subseteq \cl(T)$, so $S\cup T \subseteq \cl(S)\cup\cl(T)$ and hence $\cl(S\cup T) \subseteq \cl(\cl(S)\cup\cl(T))$ by monotonicity.  For the reverse containment, first use monotonicity to get $\cl(S) \subseteq \cl(S\cup T)$ and $\cl(T) \subseteq \cl(S\cup T)$, hence $\cl(S) \cup \cl(T) \subseteq \cl(S \cup T)$, then take the closure of both sides and apply the fact that closure is an idempotent operation.

Next, for any circuit $C$ and any $j\in C$ we claim that $\cl(C) = \cl(C-j)$.  Monotonicity implies the inclusion $\cl(C-j) \subseteq \cl(C)$, so to show the opposite inclusion, suppose $i\in E$ satisfies $\rk(C\cup i ) = \rk(C)$.  The fact that $C$ is a minimal dependent set implies $\rk(C) = \rk(C-j)$, so monotonicity of the rank function gives
\[\rk(C\cup i - j) \le \rk(C\cup i) = \rk(C) = \rk(C-j) \le \rk(C\cup i - j),\] hence $\rk(C \cup i - j) = \rk(C - j)$ and so $i\in \cl(C-j)$.  This shows that the congruence defining $Q_M$ is contained in the congruence discussed in the preceding paragraph.  Thus, by transitivity, it only remains to prove that for any $S\subseteq E$, in $Q_M$ we have \[\sum_{i \in S} x_i = \sum_{i\in \cl(S)} x_i.\] Write $\cl(S) - S = \{s_1, \ldots, s_m\}$. By the lemma, for each $s_j$ there exists a circuit $C_j$ such that $s_j \in C_j \subseteq S \cup s_j$.
	Then $\sum_{i \in C_j-s_j} x_i = \sum_{i \in C_j} x_i$ by our description of the congruence defining $Q_M$.  Due to idempotency, adding $\sum_{i \in S} x_i$ to both sides yields $\sum_{i \in S} x_i = \sum_{i \in S \cup s_j} x_i$.  Summing these latter equalities over all $j=1,\ldots,m$ establishes the desired equality.
\end{proof}

The tropical linear space $L_M$ also has a nice matroidal interpretation:
\begin{theorem}[{\cite[Theorem 3.8]{murota01}}] \label{theorem: L-cocycles}
	The submodule $L_M \subseteq \B^E$ is generated by $\sum_{i \in K} e_i$,
	where $K$ ranges over all cocircuits of $M$.
\end{theorem}

(See also \cite[Proposition 4.1.9]{frenk13} and \cite[\S4.4]{giansiracusa17}.)  Indeed, these authors establish that \[g_A := \sum_{i\in E-A}[M]_{A\cup i}e_i,\] for $A \in \binom{E}{d-1}$, span $L_M$, and similarly to Lemma \ref{lemma: circuits} we can check that these are the indicator vectors of the cocircuits:
\begin{lemma}\label{lemma: cocircuits}
If $A\in \binom{E}{d-1}$ and $g_A \ne 0$, then $g_A = \sum_{i\in K}e_i$ for some cocircuit $K$ of $M$, and conversely for any cocircuit $K$ of $M$ we have $\sum_{i\in K}e_i = g_A$ for some $A\in \binom{E}{d-1}$.
\end{lemma}

\begin{proof}
If $A\in \binom{E}{d-1}$ and $g_A \ne 0$, then $B := A\cup e$ is a basis for some $e\in E-A$.  Let \[K := C_{M^*}(e,E-(A\cup e))\] be the fundamental cocircuit of $e$ with respect to $A\cup e$.  We claim that $g_A = \sum_{i\in K}e_i$, i.e., $A\cup i$ is a basis, for $i\in E-A = (E-B)\cup e$, if and only $i\in K$.  For $i=e$ this is trivial, and by \cite[\S2.2 Exercise 10(b)]{oxley11} for $i\in E-B$ we have $i\in K$ if and only if $e\in C(i,A\cup e)$; this readily implies the claim.  For the converse, since any cocircuit $K$ is a fundamental cocircuit, say $K=C_{M^*}(e,E-B)$, we can set $A = B-e$ and then $g_A = \sum_{i\in K}e_i$.
\end{proof}

\subsection{Algebraic duality}

In \cite[Remark 4.3.3]{giansiracusa17} it is conjectured that a duality holds between a certain submodule and quotient module.  A special case of this is the assertion $Q_M = L_M^\vee$.  This is known by some experts, but we present here a proof based on our matroidal interpretation of these modules (recall from Equation \eqref{eq:Qdual} that $L_M = Q_M^\vee$ holds essentially by definition):

\begin{theorem} \label{theorem: algebraic-duality}
	We have $Q_M = L_M^\dual$.
\end{theorem}
\begin{proof}
	Since $L_M = Q_M^\dual$,
	our goal is to show that the natural map $\lambda: Q_M \to Q_M^{\dual\dual}$ is an isomorphism.
	Let $V := \Bdual{E}$.
	If $q: V \surj Q_w$ denotes the quotient map, then the diagram
	\begin{equation*}
		\xymatrix{ 
			V  \ar[r]^\cong \ar@{->}[d]^q &  V^\ddual \ar@{->}[d]^{q^\ddual} \\
			Q_M \ar[r]^\lambda &  Q_{M}^\ddual }
	\end{equation*}
	commutes by naturality.  The dual of a surjective map is injective, so $q^\vee$ is injective.  By \cite{wang94}, $\B$ is an injective $\B$-module and so the dual of any injective map of $\B$-modules is surjective (see also \cite[Lemma 3.2.1(ii)]{frenk13}).  Thus $q^\ddual$ is surjective, so by commutativity $\lambda$ is also surjective.

	To show $\lambda$ is injective, suppose $f=\sum_{i \in F} x_i$ and $f'=\sum_{i \in F'} x_i$ are distinct elements of $Q_M$.  We may assume $F$ and $F'$ are distinct flats of $M$ by Theorem \ref{theorem: Q-is-flats};
	without loss of generality, $F' \not\subseteq F$.  Complements of flats are precisely the unions of cocircuits, and by Theorem \ref{theorem: L-cocycles} the indicator vectors of unions of cocircuits are precisely the elements of $L_M$, so if we set $K := E-F$ and $g_K := \sum_{i \in K} e_i$, then $g_K \in L_M$.
	Viewing $L_M$ as $Q_M^\dual$, we see that $g_K(f) = 0$ since $K$ and $F$ are disjoint, while $g_K(f') = 1$ since $K$ and $F'$ intersect nontrivially and addition is idempotent.  Thus $\lambda(f) \neq \lambda(f')$, as desired.
\end{proof}

\begin{remark}
	While the quotient module $Q_M$ had not been introduced at the time Frenk wrote his thesis, portions of the results in this subsection and the previous one are in \cite{frenk13}.  Indeed, in \cite[Proposition 4.1.5]{frenk13} he shows that $L_M$ is  isomorphic as a lattice to $\mathcal{L}(M)$, and in \cite[Remark on p.89]{frenk13} he shows that $L^\vee$ is in bijection with $\mathcal{L}(M)$.  Note that $L_M$ is isomorphic to $\mathcal{L}(M)$ with the operation of intersection while $Q_M$ is isomorphic to it with the operation of join.
\end{remark}

\subsection{Intermediate wedge powers}

The quotient module $Q_M$ is the lattice of flats $\mathcal{L}(M)$ (Theorem \ref{theorem: Q-is-flats}), and the top wedge power $\bigwedge^d Q_M$ is isomorphic to $\B$ \cite[Theorem 6.1.1]{giansiracusa17} and hence is trivially a matroid lattice.  A natural question then is whether for any $1 \le k \le d$ there exists a matroid $N$ on the ground set $\binom{E}{k}$ such that the quotient $\Bdual{\binom{E}{k}} \twoheadrightarrow Q_N$ factors as \[\Bdual{\binom{E}{k}} \cong \bigwedge^k\Bdual{E} \twoheadrightarrow \bigwedge^k Q_M \cong Q_N,\] We conclude this section by showing that this is not the case in general.

\begin{example}
	Let $M = U_{3,6}$ be the uniform rank 3 matroid on $E = \{1,2,3,4,5,6\}$.  The circuits are $\binom{E}{4}$, so by Lemma \ref{lemma: wedge-power-relations}, 
	$\bigwedge^2 Q_M$ is generated by the wedge monomials $x_{ij} := x_i \wedge x_j$ subject to the bend relations of $x_i \wedge (x_j + x_k + x_\ell + x_m)$.
	These relations are generated by just those where $i \in \{j,k,\ell,m\}$,
	so the relations are generated by $\bend{x_{ij} + x_{ik} + x_{i\ell}}$ for all $\{i,j,k,\ell\}\in\binom{E}{4}$. 

Suppose there exists a matroid $N$ such that $\bigwedge^2 Q_M \cong Q_N$ as discussed above.  First, we claim that $N$ has rank at most 3.  Indeed, since $\cl(\{1,2,3\}) = E$ we know that $x_1+x_2+x_3 = \sum_{i=1}^6 x_i$ in $Q_M$ and so in $\bigwedge^2 Q_M$ we have
\[x_{12} + x_{13} + x_{23} = (x_1 + x_2 + x_3) \wedge (x_1 + x_2 + x_3) = \left(\sum_{i=1}^6 x_i\right)\wedge \left(\sum_{i=1}^6 x_i\right) = \sum_{1 \le i < j \le 6} x_{ij},\] hence $\rk(N) \le 3$ by the interpretation of $\bigwedge^2 Q_M \cong Q_N$ as $\mathcal{L}(N)$.  Next, we claim that $x_{12} + x_{34} + x_{56}$ and all of its subsums are distinct.  This follows from the observation that all non-trivial relations in the congruence generated by the $\bend{x_{ij} + x_{ik} + x_{i\ell}}$ have at least two variables with a common index.  This means that all subsets of $\{12,34,56\}$ are flats of $N$, so $\{12,34,56\}$ is a flat of rank 3.  But since $N$ itself then has rank 3, by the flats interpretation of $\bigwedge^2 Q_M$ we must have  \[x_{12} + x_{34} + x_{56} = \sum_{1 \le i < j \le 6} x_{ij}\] whereas by the explicit description of the relations in $\bigwedge^2 Q_M$ we know that this equality does not hold.  This contradiction shows that no such matroid $N$ exists.
\end{example}

\begin{remark}
It remains an open question to find a matroidal interpretation of the intermediate wedge powers of the quotient $\Bdual{E} \twoheadrightarrow Q_M$.
\end{remark}

\section{Matroid minors}

For a subset $T\subseteq E$ of the ground set of the matroid $M$, the \emph{deletion} $M\backslash T$ is by definition the matroid restriction $M|(E-T)$.  The \emph{contraction} can be defined in terms of matroid duality, namely $M/T = (M^*\backslash T)^*$.  These are both matroids on the ground set $E-T$.  Concretely, if $\mathcal{B} \subseteq \binom{E}{d}$ denotes the set of bases of $M$, then the bases of the deletion $M\backslash T$ are the maximal members of $\{B-T : B\in \mathcal{B}\}$ and the bases of the contraction $M/T$ are the subsets $B' \subseteq E-T$ such that $B'\cup B_T \in \mathcal{B}$, where $B_T$ is any fixed basis of the restriction $M|T$ \cite[Corollary 3.1.8]{oxley11}.  The deletion and contraction operations commute (with each other and with themselves) \cite[Proposition 3.1.25]{oxley11} and a \emph{minor} is any iteration of these operations.

The definition of deletion and contraction, and hence of minors, has been generalized from matroids to valuated matroids \cite{dress92,murota10}; the effect of this on the associated tropical linear spaces was described by Frenk in \cite[Lemma 4.1.11]{frenk13}.  Over $\B$, Frenk's description specializes to the following, which by Theorem \ref{theorem: L-cocycles} is an algebraic relation between the cocircuits of a matroid and the cocircuits of its minors.  We include here a short matroid-theoretic proof.

\begin{theorem} \label{theorem: L-minors}
	Let $T \subseteq E$ be a subset of the ground set of the matroid $M$.
	\begin{enumerate}[(a)]
	\item	The tropical linear space associated to the contraction $M/T$ is 
		\begin{equation*}
			L_{M/T} = L_M \cap \B^{E-T},
		\end{equation*}
		where the inclusion $\B^{E-T} \subseteq \B^E$ is induced by the inclusion $E-T\subseteq E$.
	\item The tropical linear space associated to the deletion $M\backslash T$ is
		\begin{equation*}
			L_{M\backslash T} = \pi_{E - T} \left(L_M\right),
		\end{equation*}
		where $\pi_{E-T}$ is projection onto the coordinate subspace $\B^{E-T} \subseteq \B^E$.
	\end{enumerate}
\end{theorem}
\begin{proof}
By Theorem \ref{theorem: L-cocycles}, $L_{M/T}$ is generated by the indicator vectors of the cocircuits of $M/T$.
		 	By \cite[3.1.17]{oxley11}, 
			the cocircuits of the contraction $M/T$ are the cocircuits of $M$ disjoint from $T$.
		For cocircuits $K_i$ of $M$ with indicator vectors $g_{K_i} := \sum_{j\in K_i}e_j \in L_M$, we have 
		$\sum_{i} g_{K_i} \in L_M \cap \B^{E-T}$ if and only if
		the coefficient of $e_t$ in $g_{K_i}$ is zero for every $i$ and $t \in T$, since $\B^E$ is a free $\B$-module.
		Thus, $L_M \cap \B^{E-T}$ is also generated by the indicator vectors of the cocircuits of $M/T$.  This proves (a).

For (b), we first prove that if $K$ is a cocircuit of $M$, then $K-T$ is a union of cocircuits of $M\backslash T$.
		Since unions of cocircuits are the complements of flats,
		it suffices to show that if $F$ is a flat of $M$, then $F-T$ is a flat of $M\backslash T$.
		Being a flat is equivalent to equaling the closure, and by \cite[3.1.15]{oxley11} the closure functions are related by $\cl_{M\backslash T}(S) = \cl_{M}(S) - T$.
		Hence \[\cl_{M\backslash T}(F-T) = \cl_M(F-T) -T = F-T,\]
		since \[F-T \subseteq \cl_M(F-T)-T \subseteq \cl_M(F) -T= F-T.\]
		(Monotonicity of the closure is \cite[Lemma 1.4.3]{oxley11}.) Thus, $\pi_{E-T}(L_M) \subseteq L_{M\backslash T}$.
		Conversely, by \cite[3.1.16]{oxley11},
		the cocircuits of $M\backslash T$ are the minimal sets $K-T$ for $K$ a cocircuit of $M$.
		Thus, $\pi_{E-T}(L_M)$ contains the generators of $L_{M\backslash T}$.
\end{proof}

Minors, as with many matroid constructions, can be defined in terms of various cryptomorphisms.
Theorem \ref{theorem: L-minors} provides an algebraic formulation of minors in terms of cocircuits,
but given the importance of tropical Pl\"ucker vectors (for us, matroid basis indicator vectors),
an algebraic description in terms of bases would be convenient.  This is complicated in general, but for the deletion or contraction by a single element there is a nice statement:

\begin{proposition}\label{prop: one-minor-plucker}
For $t\in E$, the unique expression \[[M] = w\wedge e_t + w'\] with $w \in \bigwedge^{d-1} \B^{E-t}$ and $w' \in \bigwedge^d \B^{E-t}$ is: 
\[[M] = \begin{cases} 
[M/t]\wedge e_t + [M\backslash t] &\mbox{if } w \ne 0 \text{ and } w' \ne 0 \\ 
[M\backslash t]\wedge e_t  = [M/t]\wedge e_t & \mbox{if } w'=0 \\
[M\backslash t] = [M/t] & \mbox{if } w=0
\end{cases}.\]
\end{proposition}

\begin{proof}
By \cite[3.1.14]{oxley11}, the bases of $M\backslash t$ are the maximal sets of the form $B-t$ for bases $B$ of $M$.  If $w'\ne 0$, then there exists bases of $M$ that do not contain $t$ (these are precisely the bases corresponding to $w'$) and for all such $B$ we have that $B-t = B$ has size $d$, whereas for each basis $B$ corresponding to $w\wedge e_t$ (i.e., those containing $t$) we have $|B-t| = d-1$.  Since all bases of a matroid have the same size, it must then be that $[M\backslash t] = w'$.  On the other hand, if $w'= 0$ then every basis of $M$ contains $t$ so the sets of the form $B-t$ for bases $B$ of $M$ are precisely those corresponding to $w$; since they all have the same size, $d-1$, they are all maximal, hence $[M\backslash t] = w$.  Finally, if $w \ne 0$, then there exists a basis of $M$ that contains $t$, so $[M|t] = e_t$.  By \cite[Corollary 3.1.8]{oxley11}, the bases of $[M/t]$ are then all the subsets $S\subseteq E-t$  such that $S\cup e_t$ is a basis of $M$, so $w = [M/t]$.
\end{proof}

\section{Strong maps of matroids}

Here we describe strong maps of matroids in terms of the modules $Q_M$ and $L_M$.  This leads to a faithful embedding of the category of matroids with strong maps into a category of $\B$-module homomorphisms.  While this embedding is not full, we are able to reverse-engineer it and uncover a combinatorially natural extension of the notion of strong map that gives a category of matroids that does embed fully faithfully in our module-theoretic category.

\subsection{Algebraic characterization of strong maps}\label{subsec:strong}

A \textit{strong map} $f: M \to N$ between matroids $M$ and $N$ on ground sets $E$ and $F$
is a function $f: E\sqcup \bullet \to F\sqcup\bullet$, with $f(\bullet) = \bullet$, such that if $M_\bullet$ and $N_\bullet$ are the extensions of $M$ and $N$ to $E\sqcup\bullet$ and $F\sqcup\bullet$ with $\bullet$ a loop, 
then the preimage of every flat of $N_\bullet$ is a flat of $M_\bullet$ (cf. \cite[Definition 8.1]{white86}).  
This condition translates nicely into our algebraic framework.  First, some notation: associated to any set-theoretic map $f : E\sqcup\bullet \to F\sqcup\bullet$ is the $\B$-module homomorphism \[f_* : \Bdual{E} \to \Bdual{F}\] determined by 
\begin{equation*}
	f_*(x_i) =\begin{cases} x_{f(i)} &\mbox{if }  f(i) \neq \bullet \\
			0	&\mbox{if }  f(i) = \bullet
		\end{cases}
\end{equation*}
for all $i \in E$.  Consequently, there is a linearly dual map $f_*^\vee : \B^F \rightarrow \B^E$.

\begin{proposition} \label{prop: L-Q-strong-maps}
	For a function $f : E \sqcup \bullet \to F\sqcup\bullet$ with $f(\bullet) = \bullet$, the following are equivalent:
	\begin{enumerate}
		\item $f$ defines a strong map $M \to N$ of matroids; \label{condition:strong}
		\item $f_\ast$ induces a commutative square of $\B$-modules
			\begin{equation*}
				\xymatrix{ \Bdual{E} \ar@{->>}[d] \ar[r]^{f_\ast} & \Bdual{F}\ar@{->>}[d] \\ Q_M \ar[r] & Q_{N}; }
			\end{equation*} \label{condition:Q}
		\item $f_\ast^\dual(L_N) \subseteq L_M$. \label{condition:L}
	\end{enumerate}
\end{proposition}
\begin{proof}
	\Ifandonlyif{condition:strong}{condition:L}: By the definitions of $f_*$ and linear duality, for any $k\in F$ we have \[f_\ast^\dual(e_k) = \sum_{i \in f^{-1}(k)} e_i.\]
	By Theorem \ref{theorem: L-cocycles}, the condition $f_\ast^\dual(L_N) \subseteq L_M$ is then equivalent to the condition that the preimage under $f$ of every union of cocircuits of $N$ is a union of cocircuits of $M$.  
	Since $\bullet$ is a loop in $M_\bullet$,
	the cocircuits of $M_\bullet$ are the same as those of $M$,
	and similarly for $N$.
	The result then follows from the fact that unions of cocircuits are exactly the complements of flats.

	\Ifandonlyif{condition:Q}{condition:L}: This is an immediate consequence of the duality established in Theorem \ref{theorem: algebraic-duality}.  Indeed, the condition $f_\ast^\dual(L_N) \subseteq L_M$ is equivalent to having the  commutative square
	\begin{equation*}
		\xymatrix{
			L_N \ar[r] \ar@{^{(}->}[d] &L_M \ar@{^{(}->}[d]		\\
			\B^F \ar[r]^{f_\ast^\dual}	&	\B^E,					
		}
	\end{equation*}
and the functor $\Hom(-,\B)$ transforms this square into the one in the statement of (2) and vice versa.
\end{proof}

\begin{remark}
We have not found a module-theoretic characterization of weak matroid maps.
\end{remark}

\subsection{Matroid quotients}

Here we specialize to the case that the matroids $M$ and $N$ have the same ground set $E$ and only consider the identity function $\id : E_\bullet \rightarrow E_\bullet$.  In this case a strong map is called a matroid \emph{quotient}, so Proposition \ref{prop: L-Q-strong-maps} specializes to the following (cf., \cite[Lemma 2.21]{shaw13}):

\begin{corollary} \label{cor: L-Q-strong-identity}
	The following are equivalent:
	\begin{enumerate}
		\item $N$ is a matroid quotient of $M$;
		\item the quotient map $\Bdual{E} \twoheadrightarrow Q_N$ factors through the quotient map $\Bdual{E} \twoheadrightarrow Q_M$;
		\item $L_N \subseteq L_M$.
	\end{enumerate}
\end{corollary}

A \emph{flag matroid}, as introduced in \cite{BGW03} (see also \cite{FlagMatroids}), is equivalent to a sequence of matroid quotients \cite[Theorem 1.7.1]{BGW03}, so by this corollary it is also equivalent to a flag (in the usual geometric sense of the term) of constant coefficient tropical linear spaces.

The fundamental relationship between strong maps and minors
is expressed in the factorization theorem for strong maps, which states that every strong map factors as an embedding  followed by a contraction \cite[Theorem 8.2.7]{white86}.  In the special case we are considering here of $\id : E_\bullet \rightarrow E_\bullet$, this takes on the following formulation:

\begin{proposition}[{\cite[Proposition 7.3.6]{oxley11}}]\label{prop:matroidquot}
	$N$ is a matroid quotient of $M$ if and only if there exists a matroid $P$ on $E \sqcup T$ with $|T| = r(M) - r(N)$
	such that $P \backslash T = M$ and $P/T = N$.
\end{proposition}

Combining this result with Corollary \ref{cor: L-Q-strong-identity} and Theorem \ref{theorem: L-minors} yields the following:

\begin{corollary} \label{corollary: linear-space-inclusion}
Let $L$ and $L'$ be tropical linear spaces in $\B^E$.  Then $L \subseteq L'$ if and only if there exists a tropical linear space $L''\subseteq \B^{E\sqcup T}$ with $|T| = \dim(L')-\dim(L)$ such that $L= L'' \cap \B^E$ and $L' = \pi_{E}(L'')$.\end{corollary}

Here $\B^E\subseteq \B^{E\sqcup T}$ is a coordinate subspace, $\pi_E : \B^{E\sqcup T} \to \B^E$ is coordinate projection, and the dimension of a tropical linear space is the rank of the corresponding matroid.

\begin{remark}
	This is the tropical-linear analogue
	of the following linear-algebraic statement: for $k$ a field,
	if $U \subseteq V \subseteq k^n$ are subspaces and $m = \dim V - \dim U$,
	then there exists a subspace $W \subseteq k^{n+m}$ with coordinate projection onto, and intersection with, $k^n$ equal to $V$ and $U$, respectively. The construction of $W$ here is easier: let $\{u_1, \ldots, u_\ell\}$ be a basis for $U$,
	and extend it by $\{v_1, \ldots, v_m\}$ to a basis for $V$.
	Say $\{e_1,\ldots, e_{n+m}\}$ is a basis for $k^{n+m}$ such that our choice of $k^n$ is spanned by $\{e_1, \ldots, e_n\}$.
	Then we can let $W$ be the span of $\{u_1, \ldots, u_\ell, v_1 + e_{n+1}, v_2 + e_{n+2}, \ldots, v_m + e_{n+m}\}$.
\end{remark}

\begin{remark}
The tropical linear space $L''$ in this corollary is a constant-coefficient tropical modification of $L'$ along $L$ (cf., \cite[\S2.4]{shaw13}).  This $L''$ is not unique.  Indeed, the matroid quotient in Proposition \ref{prop:matroidquot} can always be decomposed into a sequence of corank one quotients \cite[Proposition 7.3.5]{oxley11} and different decompositions correspond to different sequences of tropical modifications.  The relation between different decompositions and the existence of canonical decompositions was an active topic of investigation in the combinatorics community (see \cite{higgs} and the references therein).  On the other hand, over $\T$ the space of tropical modifications leads to the Berkovich analytification \cite{payne}.  Perhaps there is an interesting connection to be made here.
\end{remark}

\subsection{Categories of matroids}

Matroids form a category with strong maps as morphisms.
This category is investigated in \cite{alhawary01} and \cite{heunen17},
the latter of which discusses the existence and nonexistence of limits and colimits.
For instance, the category of matroids with strong maps does not have categorical products.  From the algebraic perspective a natural category to consider instead is the following:
\begin{definition}
Let $\mathcal{C}$ denote the following category:
\begin{itemize}
\item \emph{Objects}: surjections $\Bdual{E} \surj Q_M$ for $M$ a matroid on the ground set $E$;
\item \emph{Morphisms}: commutative squares of $\B$-module homomorphisms
\[\xymatrix{\Bdual{E}  \ar[r] \ar@{->>}[d] &\Bdual{F} \ar@{->>}[d] \\ Q_M \ar[r] & Q_N.}\]
\end{itemize}
Let $\iota(M) := (\Bdual{E}\surj Q_M)$.
\end{definition}

\begin{corollary}
	The map $\iota$ is functorial and gives a faithful, but not full, embedding of the category of matroids with strong maps into $\mathcal{C}$.
\end{corollary}

\begin{proof}
That $\iota$ is injective on objects is \cite[Proposition 6.2.1]{giansiracusa17}; that it is functorial is Proposition \ref{prop: L-Q-strong-maps}; that it is injective on Hom sets is clear, since the association $f \mapsto f_*$ from set maps $E \rightarrow F$ to $\B$-module homomorphisms $\Bdual{E} \rightarrow \Bdual{F}$ is injective; that it is not surjective on Hom sets is also clear, since the only $\B$-module homomorphisms $\Bdual{E} \rightarrow \Bdual{F}$ of the form $f_*$ are those that send basis vectors to basis vectors.
\end{proof}

\begin{remark}
By algebraic duality (Theorem \ref{theorem: algebraic-duality}) we also have a contravariant embedding $M \rightsquigarrow (L_M \subseteq \B^E)$ into the category whose objects are embedded tropical linear spaces and morphisms are $\B$-modules homomorphisms of the ambient free modules that restrict to $\B$-module homomorphisms of the embedded tropical linear space submodules.
\end{remark}

While the ``extra'' morphisms included in $\mathcal{C}$ are natural from a module-theoretic perspective, it turns out they are also rather natural from a combinatorial perspective.  Let us consider multivalued set functions $f : E \rightarrow F$ that are allowed to send elements to the empty set (i.e., functions $f : E \rightarrow 2^F$), and for such a multivalued function define the \emph{preimage} of a subset $S \subseteq F$ to be \[f^{-1}(S) := \{ i \in E: f(i) \subseteq S\}.\]

\begin{definition}
A \emph{multivalued strong map} of matroids $M \rightarrow N$ on ground sets $E$ and $F$ is a multivalued function $f : E \to F$ such that the preimage of every flat is a flat.
\end{definition}

\begin{theorem}
The map $\iota$ is a functor defining a fully faithful embedding from the category of matroids with multivalued strong maps into $\mathcal{C}$.
\end{theorem}

\begin{proof}
There is a natural bijection between multivalued set maps $f : E \to F$ and $\B$-module homomorphisms $\varphi : \Bdual{E} \to \Bdual{F}$.  Indeed, given $\varphi$ we let $f(i) = \supp \,\varphi(x_i)$, and conversely given $f$ we let $\varphi(x_i) = \sum_{j \in f(i)} x_j$.  To show that $f$ is a multivalued strong map $M \to N$ if and only if $\varphi$ is a morphism in our category, it suffices by algebraic duality to show that $f$ is a multivalued strong map if and only if $\varphi^\dual(L_N) \subseteq L_M$.  Now, for an arbitrary subset $K \subseteq F$, we have
	\begin{equation*}
		\varphi^\dual\left( \sum_{i \in K} e_i\right) = \sum_{i \in K} \sum_{\substack{j \\f(j) \ni i}} e_j,
	\end{equation*}
	that is, 
	\begin{equation*}
		\supp\, \varphi^\dual\left( \sum_{i \in K}e_i\right) = \{j \in E: f(j) \cap K\neq \varnothing\}.
	\end{equation*}
	Hence, 
	\begin{equation*}
		E - \supp\, \varphi^\dual\left(\sum_{i \in K}e_i\right) = \{j \in E: f(j) \subseteq F-K\} = f^{-1}(F-K).
	\end{equation*}
	Since the supports of elements of $L_M$ and $L_N$ are exactly the complements of flats,
	we see that $\varphi^\dual(L_N) \subseteq L_M$ if and only if the preimage under $f$ of a flat is a flat.
\end{proof}

One advantage of embedding matroids in our category $\mathcal{C}$ of quotient modules is that more limits exist.  For instance, we have finite products and coproducts in $\mathcal{C}$, and they coincide, because the same is true of the category of $\B$-modules.  The finite (co)product is the direct sum.  Explicitly, if $M$ and $N$ are matroids on ground sets $E$ and $F$, then the direct sum of $\iota(M)$ and $\iota(N)$ in $\mathcal{C}$ is the quotient $\Bdual{E}\oplus\Bdual{F} \twoheadrightarrow Q_M\oplus Q_N$.  Fortunately this categorical use of the term is compatible with the matroidal one---by definition, the \emph{direct sum} matroid $M\oplus N$ is the matroid on $E\sqcup F$ whose bases are the unions of a basis of $M$ with a basis of $N$, and we have:

\begin{proposition}
The functor $\iota$ commutes with direct sums: $\iota(M)\oplus \iota(N) \cong \iota(M\oplus N)$.
\end{proposition}

\begin{proof}
We have an isomorphism $Q_M\oplus Q_N \cong Q_{M\oplus N}$ by Theorem \ref{theorem: Q-is-flats} because the flats of $M\oplus N$ are the unions of a flat of $M$ with a flat of $N$ \cite[4.2.16]{oxley11}, and this fits into the square
\[\xymatrix{\Bdual{E}\oplus\Bdual{F}  \ar[r]^-\sim \ar@{->>}[d] &\Bdual{E\sqcup F} \ar@{->>}[d] \\ Q_M\oplus Q_N \ar[r]^-\sim & Q_{M\oplus N}}\]
which clearly commutes.
\end{proof}

Given a matroid $M$ on the ground set $E$, there is an equivalence relation on $E$ where $a\sim b$ if and only if $a=b$ or $\{a,b\}$ is contained in a circuit.  The equivalence classes $T_1,\ldots, T_m$ are the \emph{connected components} of $M$, and $M = M|T_1\oplus \cdots \oplus M|T_m$ \cite[Corollary 4.2.9]{oxley11}.

\begin{corollary}
The matroid $M$ is connected if and only if $\iota(M)$ is indecomposable in $\mathcal{C}$.
\end{corollary}

Another advantage of embedding matroids in $\mathcal{C}$ is that the set of maps between two matroids inherits the structure of a $\B$-module.  For instance:

\begin{proposition}
	The free matroid $U_{1,1}$ represents the tropical linear space of cocycles:
	\begin{equation*}
		\Hom_\mathcal{C}\left(\iota(M), \iota(U_{1,1})\right) = L_M.
	\end{equation*}
\end{proposition}
\begin{proof}
	Since $\iota(U_{1,1})$ is the identity quotient presentation $\B^\dual \to \B^\dual$,
	the isomorphism $\B^\dual \cong \B$ shows that commutative squares from $\Bdual{E} \surj Q_M$ to $\B^\dual \to \B^\dual$
	are naturally in bijection with morphisms $Q_M \to \B$.
	But these are the elements of $Q_M^\dual = L_M$.
\end{proof}

\section{Matroid union and stable sum}

Given matroids $M$ and $N$ on the ground set $E$, the \emph{matroid union} $M\vee N$ is the matroid whose independent sets are the unions of an independent set of $M$ with an independent set of $N$.  When the ground sets of $M$ and $N$ are disjoint, this is just the matroid direct sum: $M\vee N = M \oplus N$.  A much milder disjointedness condition that is germane is the following:

\begin{definition}
We say that $M$ and $N$ are \emph{sufficiently disjoint} if there exists a basis $B_M$ of $M$ and a basis $B_N$ of $N$ such that $B_M\cap B_N = \varnothing$.
\end{definition}

It is immediate from the definitions that $\rk(M\vee N) \le \rk M + \rk N$ with equality if and only if $M$ and $N$ are sufficiently disjoint.  In the exterior algebra $\bigwedge \B^E$ we have $[M]\wedge [N] \ne 0$ if and only if $M$ and $N$ are sufficiently disjoint, and in that case \cite[Proposition 5.1.1]{giansiracusa17}:
\[[M]\wedge [N] = [M \vee N] \in \bigwedge^{\rk M + \rk N}\B^E.\]

When $M$ and $N$ are sufficiently disjoint, the tropical linear space $L_{M\vee N} \subset \B^E$ is the \emph{stable sum} of $L_M$ and $L_N$, which we denote $L_M+^{st} L_N$.  This operation is studied in \cite{fink15} (see also \cite[Proposition 4.1.14]{frenk13} and \cite{murota97}).  It is tropically dual to the stable intersection of tropical linear spaces that Speyer introduced in \cite{speyer08} for valuated matroids supported on uniform matroids.  

\begin{remark}
	We can extend the exterior algebra formalism of the matroid union to the not-necessarily sufficiently disjoint setting as follow.  For a matroid $M$, let $T_i(M)$ denote the rank $i$ \emph{truncation}, which is the matroid whose independent sets are the independent sets of $M$ of size at most $i$ (cf., \cite[Proposition 7.3.10]{oxley11}).
	Then 
	\begin{equation*}
		\sum_{i \ge 0} [T_i(M)] \in \bigwedge \B^E
	\end{equation*}
	is the indicator vector of all the independent sets of $M$. 	It is then straightforward to see that 
	\begin{equation*}
		\left( \sum_{i \ge 0} [T_i(M)] \right) \wedge \left(\sum_{i \ge 0} [T_i(N)]\right) \in \bigwedge\B^E
	\end{equation*}
	is the indicator vector of all independent sets of the matroid union $M \vee N$.
\end{remark}

Since it is convenient to be able to phrase statements about the stable sum in geometric terms, without having to translate back to matroids, let us introduce the following terminology:
\begin{definition}
We say that two tropical linear spaces in $\B^E$ are \emph{sufficiently moveable} if the corresponding matroids on $E$ are sufficiently disjoint.
\end{definition}
The reason for this wording is that when matroids $M$ and $N$ are sufficiently disjoint, the tropical linear spaces $L_M\otimes \T$ and $L_N\otimes \T$ in $\T^E = (\R\cup\{-\infty\})^E$ can be perturbed so that they meet transversely and then the stable sum is the limit of the spans as the perturbation goes to zero \cite{speyer08,fink15}---but when $M$ and $N$ are not sufficiently disjoint the tropical linear spaces intersect at $-\infty$ too much to be able to perturb them to obtain transversality.

The following lemma illustrates part of how stable sum interacts with minors and strong maps. The hypotheses on the spaces involved are asymmetrical: if $T$ is contracted or deleted,
then the lemma deals with taking the stable sum with a tropical linear space supported on the complement of $T$.
This is necessary since, e.g., even if $L_{M}$ and $L_{N}$ are sufficiently moveable, we have \[([M] \wedge e_t)\wedge ([N]\wedge e_t) = [M] \wedge [N] \wedge e_t \wedge e_t = 0,\]
so $L_{[M]\wedge e_t}$ and $L_{[N]\wedge e_t}$ are \emph{not} sufficiently moveable.  A word on notation: when taking stable sum of tropical linear spaces in different ambient spaces, we implicitly rely on the natural coordinate inclusion maps.

\begin{lemma} \label{lemma: minor-stable-sum}
	Given sets $E$ and $T$, let $M$ be a matroid on $E$ and $N$ a matroid on $E\sqcup T$. Denote by $\pi_E: \B^{E\sqcup T} \to \B^E$ the coordinate projection map and $\B^E \subseteq \B^{E\sqcup T}$ the coordinate subspace.
	\begin{enumerate}[(a)]
		\item If $L_M$ and $L_N \cap \B^E$ are sufficiently moveable, then so are $L_M$ and $L_N$ and we have
			\begin{equation*}
				L_M \stsum (L_N \cap \B^E) = (L_M \stsum L_N) \cap \B^E.
			\end{equation*}
		\item If $L_M$ and $\pi_E(L_N)$ are sufficiently moveable, then so are $L_M$ and $L_N$ and we have 
			\begin{equation*}
				L_M \stsum \pi_E(L_N) = \pi_E(L_M \stsum L_N).
			\end{equation*}
	\end{enumerate}
\end{lemma}
\begin{proof}
For (a), first note that by induction on $|T|$ it suffices to consider the case that $T = \{t\}$ is a singleton. By Proposition \ref{prop: one-minor-plucker} and Theorem \ref{theorem: L-minors}, there exist elements $[N'], [N''] \in \bigwedge \B^E$ that are either 0 or basis indicator vectors for matroids on $E$ such that $[N] = [N'] \wedge e_t + [N'']$ and if $[N'] \ne 0$ then $L_{N'} = L_N \cap \B^E$ and hence $[M] \wedge [N'] \neq 0$ by hypothesis.  In this case, 
\begin{equation}\label{eq:MN}
[M] \wedge [N] = ([M] \wedge [N']) \wedge e_t + ([M] \wedge [N'']) \neq 0
\end{equation}
since addition is idempotent and $[M] \wedge [N'] \in\bigwedge \B^E$.  Thus $L_M$ and $L_N$ are sufficiently moveable, and by applying Proposition \ref{prop: one-minor-plucker} and Theorem \ref{theorem: L-minors} to the expression in \eqref{eq:MN}, and using that $[M]\wedge [N'] \ne 0$, we get 
\[L_M\stsum (L_N\cap \B^E) = L_M\stsum L_{N'} = L_{[M]\wedge [N']} = L_{[M]\wedge [N]} \cap \B^E = (L_M \stsum L_N) \cap \B^E,\] as desired.  In the case $[N'] = 0$, we have that $N=N''$ is a matroid on $E$ so $L_N \subseteq \B^E$ and the assertion in (a) is trivial.  The proof of (b) is nearly identical.
\end{proof}

We can use this lemma to prove the following monotonicity property for the stable sum:  

\begin{theorem}
Suppose that $M,N,P$ are matroids on $E$ with $L_M \subseteq L_N$.  If $L_P$ and $L_N$ are sufficiently moveable, then so are $L_P$ and $L_M$, and we have $L_P \stsum L_M \subseteq L_P \stsum L_N$.
\end{theorem}

We should first note that this stable sum monotonicity holds more generally; indeed, it is a combinatorial exercise to show that if $M' \rightarrow M$ and $N' \rightarrow N$ are arbitrary matroid quotients, then $M\vee N \rightarrow M'\vee N'$ is also a matroid quotient \cite[Exercise 8.7]{white86}.  A purely geometric proof is also possible: for tropical linear spaces in $\R^n$ (i.e., valuated matroids supported on the uniform matroid) monotonicity follows immediately from tropical duality and Speyer's interpretation of the stable intersection as a limit of perturbations \cite{speyer08}, and Speyer's framework should naturally extend to the generality involved in the above theorem.  Here we provide a module-theoretic argument.  

\begin{proof}
The hypothesis that $L_P$ and $L_N$ are sufficiently moveable means there exists a basis $B_P$ of $P$ and a basis $B_N$ of $N$ such that $B_P \cap B_N = \varnothing$.  By Theorem \ref{theorem: L-cocycles}, the hypothesis $L_M \subseteq L_N$ means every cocircuit of $M$ is a union of cocircuits of $N$.  By \cite[Proposition 2.1.19]{oxley11}, the bases of any matroid are the minimal subsets which intersect every cocircuit.  Thus, $B_N$ must contain a basis for $M$, so $B_P$ and $B_M$ are sufficiently disjoint. 
	
	By Corollary \ref{corollary: linear-space-inclusion}, the containment $L_M \subseteq L_N$ implies the existence of a set $T$ of size $\rk(N)-\rk(M)$ and a matroid $Q$ on $E\sqcup T$ such that $L_M = L_Q \cap \B^E$ and $L_N = \pi_E(L_Q)$.  By Lemma \ref{lemma: minor-stable-sum}, $L_P$ and $L_Q$ are sufficiently moveable and moreover
	\begin{equation*}
		L_P \stsum L_M = (L_P \stsum L_Q) \cap \B^E  ~\text{ and }~ L_P \stsum L_N = \pi_E(L_P \stsum L_Q).
	\end{equation*}
	Since \[\dim(L_P\stsum L_N) = \rk(P\vee N) = \rk(P)+\rk(N)\] and similarly for $M$, we have
	\[\dim(L_P\stsum L_N) - \dim(L_P\stsum L_M) = \rk(N) - \rk(M) = |T|\]
	so we can now apply Corollary \ref{corollary: linear-space-inclusion} in the reverse direction to obtain $L_P \stsum L_M \subseteq L_P \stsum L_N$.
\end{proof}


\section{Transversal matroids}

In this concluding section we illustrate how several combinatorial results in the literature on transversal matroids have appealing algebraic or geometric translations in our framework.

\begin{conv}
In this section we use the notation $[k] := \{1,2,\ldots, k\}$ for $k\in \mathbb{N}$.  To avoid confusion with the notation $[M]$ for the basis indicator vector of a matroid, we shall only use lower-case letters for natural numbers and upper-case letters for matroids.
\end{conv}

A \emph{transversal matroid} is a matroid whose independent sets are the partial transversals of a set system.  Let us take a moment to recall this terminology.  For a ground set $E$ and a sequence of subsets $A_1,\ldots,A_m \subseteq E$ (called a \emph{set system}), a \emph{transversal} is a subset $S\subseteq E$ such that there exists a bijection $\psi : [m] \to S$ with $\psi(i) \in A_i$ for all $1 \le i \le m$; a \emph{partial transversal} is a subset $S\subseteq E$ that is a transversal of a sub-sequence of the set system.  The partial transversals of any set system form the independent sets of a matroid \cite[Theorem 1.6.2]{oxley11}.  

Transversal matroids are a well-studied and important class of matroids.  Relevant for us is that they correspond to the tropical linear spaces that are stable sums of lines \cite{fink15} and to the totally decomposable multivectors in the exterior algebra \cite{giansiracusa17}.  Indeed, a matroid $M$ is transversal if and only if it is the matroid union of rank-one matroids \cite[Proposition 11.3.7]{oxley11} and in this case it can be written as the matroid union of exactly $\rk(M)$ rank-one matroids \cite[Theorem 2.6]{bonin10}.

Writing a rank $d$ matroid $M$ on the ground set $E=[n]$ as a matroid union of $d$ rank-one matroids is equivalent to factoring the basis indicator vector
\[[M] = w_1\wedge\cdots \wedge w_d \in \bigwedge^d\B^E\] 
into a wedge of $d$ vectors $w_i\in \B^E$, and by \cite[Proposition 3.1.4]{giansiracusa17} this is equivalent to choosing a $d\times n$ matrix over $\B$ whose nonzero maximal minors are precisely the bases of $M$ (that is, a $d\times d$ sub-permanent indexed by columns $I \in \binom{E}{d}$ is $1\in\B$ if and only if $I$ is a basis of $M$).  Concretely, the rows of the matrix are the vectors $w_i$ which in turn are the basis indicator vectors of the rank one matroids in the matroid union.  We summarize this situation as follows:

\begin{definition}
A \emph{$\B$-presentation} of a rank $d$ transversal matroid $M$ on $E=[n]$ is a $d\times n$ matrix over $\B$ such that the vector of maximal minors equals the basis indicator vector $[M]$.
\end{definition}

\begin{remark}
While over a field the vector of maximal minors of a full rank matrix is the same data as the row space, in the tropical setting the row space is not well-behaved (cf., \cite[\S5]{yu-yuster}) so we really must work with tropical Pl\"ucker vectors.  Indeed,  the associated tropical linear space is the stable sum of the rows, which in general properly contains the span of the rows.  Nonetheless, it is convenient to think of transversal matroids as the matroids that are representable over $\B$.
\end{remark}

\subsection{Tropical analogue of row operations}

Let us continue the analogy brought up in the preceding remark.  Over a field $k$, if we have a $d$-dimensional linear subspace of an $n$-dimensional vector space and we represent it as the row space of a full-rank $d\times n$ matrix $A$, then the full set of $d\times n$ matrices representing this same linear space is the $\GL_d$-orbit of $A$.  These orbits are the fibers of the Stiefel map $k^{d\times n} \dashrightarrow \Gr(d,n)$ sending a matrix to its row space, viewed as a point of the Grassmannian.  Gauss-Jordan elimination can thus be viewed as providing a set-theoretic section (it is discontinuous) of this map, since it specifies a unique representative of each $\GL_d$-orbit---namely, the reduced row echelon form of the matrix.  

Back in the idempotent world, we have the tropical Stiefel map 
$\B^{d\times n} \dashrightarrow \Gr^{trop}(d,n)$ sending each matrix 
with a nonzero $d\times d$ permanent to its vector of maximal minors, 
which is both a transversal matroid and a realizable tropical Pl\"ucker 
vector (hence point in the tropical Grassmannian).  This tropical 
Stiefel map more generally over $\T$ is the main topic of investigation 
in \cite{fink15}.  By restricting to $\B$ we can, in a sense, geometrize some combinatorial 
results from Bondy and Welsh in \cite{bondy-welsh,bondy72}.  For instance, let us call a $\B$-presentation \emph{maximal} (resp. \emph{minimal}) if changing any 0 to a 1 (resp. 1 to a 0) changes the corresponding transversal matroid.

\begin{theorem}
Each nonempty fiber of the rational map $\B^{d\times n} \dashrightarrow 
   \bigwedge^d\B^n$ sending a $\B$-matrix (with a nonzero maximal minor) 
   to its vector of maximal minors has a unique maximal element but in 
   general multiple distinct minimal elements.
\end{theorem}

\begin{proof}
This is an immediate translation of \cite[Theorem 1]{bondy72}.
\end{proof}

\begin{remark}
Maximal $\B$-presentations can thus be viewed as a tropical analogue of reduced row echelon form, and the process of transforming a $\B$-presentation of a transversal matroid to the maximal $\B$-presentation is a tropical analogue of Gauss-Jordan elimination.  This is a slightly odd analogy since the tropicalization of a reduced row echelon form matrix is minimal, rather than maximal, as a $\B$-presentation of the corresponding transversal matroid.  However, there are other (unrelated) instances in tropical geometry where minimality and maximality are reversed under tropicalization---for instance, the maximal strata in the moduli space of tropical curves parameterize the dual graphs of the nodal curves in the minimal strata of the classical moduli space \cite{ACP}.  
\end{remark}

\subsection{Fundamental transversal matroids}

If a rank $d$ matroid on $[n]$ is representable over a field $k$, meaning that its bases correspond to the nonzero maximal minors of a $d\times n$ matrix over $k$, then by acting by $\GL_d$ we can transform a given $d\times d$ submatrix to the identity $I_d$ without changing the matroid represented by the matrix.  Since matroid isomorphisms are permutations of the ground set, this means every matroid representable over a field is isomorphic to a matroid representable by a matrix of the form $[I_d ~ A]$.  It turns out that only a special class of transversal matroids has a $\B$-presentation of this form.  Indeed, a  \emph{fundamental transversal matroid} (also called a \emph{principal transversal matroid}) is a matroid that has a basis $B$ such that each cyclic flat $F$ (i.e., $F$ is both a flat and a union of circuits) is spanned by $F\cap B$.  A known characterization of these (see, e.g., \cite{bonin-kung}) is that a rank $d$ matroid on $[n]$ is fundamental transversal if and only if it admits a $\B$-presentation by a $d\times n$ matrix with a $d\times d$ permutation matrix as a submatrix---or in other words, these are the matroids isomorphic to ones with a $\B$-presentation of the form $[I_d ~ A]$.  Due to the idempotency of addition over $\B$, this can be expressed in the following equivalent way:

\begin{theorem}
A rank $d$ transversal matroid on $[n]$ is fundamental transversal if and only if it admits a $\B$-presentation by a matrix inducing a surjective linear map $\B^{n} \twoheadrightarrow \B^d$.
\end{theorem}

In other words, the matroids corresponding to $\B$-matrices are the transversal matroids, and the ones corresponding to surjective $\B$-matrices are the fundamental transversal matroids---in contrast to the situation over a field where if a matroid is representable then the representing matrix is always surjective (i.e., full rank).

\bibliographystyle{amsalpha}
\bibliography{references}

\end{document}